\documentclass[12pt, a4paper, reqno]{amsart}

\usepackage[cm]{fullpage}
\usepackage{amsmath, amssymb, amsthm}
\usepackage[T1]{fontenc}
\usepackage{newpxtext,newpxmath}
\usepackage[utf8]{inputenc}
\usepackage{booktabs}

\usepackage{color}
\definecolor{forestgreen}{rgb}{0.0, 0.27, 0.13}
\definecolor{ultramarine}{rgb}{0.07, 0.04, 0.56}
\definecolor{lust}{rgb}{0.9, 0.13, 0.13}

\usepackage[pagebackref=true,colorlinks=true,linkcolor=ultramarine,
citecolor=ultramarine,urlcolor=forestgreen]{hyperref}
\usepackage{indentfirst}

\newcommand{\cE}{\mathcal{E}}

\newcommand{\cQ}{\mathcal{Q}}

\newcommand{\bQ}{\mathbf{Q}}

\renewcommand{\Re}{\operatorname{Re}}

\DeclareMathOperator{\ord}{ord}

\theoremstyle{plain}
\newtheorem{theorem}{Theorem}[section]
\newtheorem{lemma}[theorem]{Lemma}
\newtheorem{proposition}[theorem]{Proposition}
\newtheorem{corollary}[theorem]{Corollary}

\theoremstyle{definition} 

\theoremstyle{remark} 
\newtheorem*{remark}{Remark}

\DeclareFontFamily{U}{wncy}{}
\DeclareFontShape{U}{wncy}{m}{n}{<->wncyr10}{}
\DeclareSymbolFont{mcy}{U}{wncy}{m}{n}
\DeclareMathSymbol{\Sha}{\mathord}{mcy}{"58}

\newcommand{\lp}{\left(}
\newcommand{\rp}{\right)}

\newcommand{\lbrb}[1]{\lp #1 \rp}
\newcommand{\lcrc}[1]{\{ #1 \}}

\newcommand{\quasym}[2]{\left( \dfrac{#1}{#2} \right)_4}
\newcommand{\quadsym}[2]{\left( \dfrac{#1}{#2} \right)_{2}}

\title{Distribution of root numbers of Hecke characters attached to
  some elliptic curves}
\author{Keunyoung Jeong, Jigu Kim, Taekyung Kim}

\address{Department of Mathematical Sciences, Ulsan National Institute
  of Science and Technology, 50 UNIST-gil, Ulsan 44919, Republic of Korea}
\address{Institute of Mathematical Sciences, Ewha Womans University, 52 Ewhayeodae-gil, Seodaemun-gu, Seoul 03760, Republic of Korea}
\address{Cryptolab, Inc., Seoul National University, 1 Gwanak-ro, Gwanak-gu, Seoul 08826, Republic of Korea } 

\keywords{Hecke characters, Complex multiplication, Distribution of root numbers}
\subjclass[2010]{Primary 11G15, Secondary 11N69}

\email{kyjeongg@gmail.com}
\email{jigu.kim.math@gmail.com}
\email{Taekyung.Kim.Maths@gmail.com}

\begin{document}

\begin{abstract}
  In this paper, we show that an action on the set of elliptic curves
  with $j=1728$ preserves a certain kind of symmetry on the local root
  number of Hecke characters attached to such elliptic curves.  As a
  consequence, we give results on the distribution of the root numbers
  and their average of the aforementioned Hecke characters.
\end{abstract}

\maketitle

\section{Introduction}

Let $K$ be the number field $\mathbf{Q}(i)$ and $O$ its ring of
integers.  We fix an embedding $K \hookrightarrow \mathbf{C}$ once and
for all.  The main object of this paper is the root number of Hecke
characters attached to elliptic curves 
over $K$ with complex multiplication by
$O$.  Unlike the case of root numbers of elliptic curves, root numbers
of Hecke characters associated with the curves need not have
value $\pm 1$, and in general they have values in the complex unit
circle.  However, we will show that there is still a ``symmetry'' on
these root numbers.

Before giving a formal definition of the symmetry, we define some
basic notions.  By the theory of complex multiplication, the elliptic
curve over $K$
with complex multiplication by $O$ is uniquely determined by an
equation
\begin{equation} \label{eqn:intro}
E_d : y^2 = x^3 - dx
\end{equation}
up to isomorphism defined over $K$ for the fourth-power-free element
$d \in O$. Let $\cE$ be the set of $K$-isomorphism classes of such
elliptic curves.  Then the group $K^\times/(K^\times)^4$ acts on $\cE$
by $\overline{x}.E_{d} := E_{xd}$.  By definition, an element $x$ of
$K^\times$ gives an action of order dividing 4, i.e.,
$\overline{x}^4.E_{d}$ is isomorphic to $E_{d}$.

We say that an action induced by $\overline{x}$ \emph{preserves a
  symmetry} on the (local) root numbers when the root numbers are
changed only by a multiple of $(2\pi)/4$ via the action.  Since 
any
root numbers of elliptic curves are either $1$ or $-1$, arbitrary
actions induced by an element of $K^\times/(K^\times)^4$ preserve a
symmetry on the root numbers of elliptic curves.  However, in the case
of Hecke characters attached to elliptic curves with complex
multiplication, it is not true in general.  Nevertheless, we will show
that units and primes of degree two (cf.~Subsection
\ref{sec:gaussian-primes}) preserve a symmetry on the root numbers of
Hecke characters. More precisely, we will compute the ratio between
the root numbers of Hecke character induced by $E_d$ and
$\overline{x}.E_d$ where $x$ is a unit or a prime of degree two.

Intuitively, the existence of such a symmetry-preserving action shows
that the local root numbers do not particularly prefer one quarter circle
than others.  By adapting this intuition, we can show that the set of
local root numbers of Hecke character of elliptic curves with complex
multiplication by $O$ at bad primes is dense in the unit circle.

\begin{theorem}
  \label{main nearball}
  Let $S^1$ be the unit circle in $\mathbf{C}$.  For arbitrary
  $\theta \in S^1$ and $\epsilon>0$, there is an elliptic curve $E$
   over $K$
  having CM by the ring $O$ and a prime $v$ of $K$ such that
  the root number $w_v(\chi_E) \in B_{\theta}(\epsilon)$ where
  $B_z(r)$ denotes the open ball having centre at $z \in \mathbf{C}$
  and radius $r>0$.
\end{theorem}

As primes of degree two preserve a symmetry, when we are given a local
root number, there are certain powers of a prime of degree two that do
not change the given root number when multiplied by the power.

\begin{theorem}
  \label{main manytheta}
  Let $\theta \in S^1$ be fixed.  If there exists a single elliptic
  curve $E_0$ over $K$
  having CM by the ring $O$ and a prime $v$ of $K$ such
  that the root number $w_v(\chi_{E_0}) = \theta$, then there are
  infinitely many (K-isomorphism classes of) elliptic curves 
  $E$ over $K$ having CM
  by $O$ such that $w_v(\chi_E) = \theta$.
\end{theorem}

Local symmetry-preserving properties also lead us into the computation
of the average value of the global root numbers.  Given a suitable
``base element'' $d \in O$, we investigate the local root numbers of
Hecke characters $\chi_{dQ}$, where $Q$ varies over the products of
distinct primes of degree two.  The symmetry-preserving property for
these kinds of primes allows us to compare the local root numbers
between $\chi_d$ and $\chi_{dQ}$.  With a help from quartic
reciprocity and the balanced distribution of quadratic residues, we
can show that the average ``global'' root number (except for the
$2$-part) with such varying $Q$ is equal to zero.

\begin{theorem}
  \label{main average}
  For a real number $X > 0$, let $\mathcal{Q}(X)$ be the set of elements $Q$ of $O$ such that
  \begin{itemize}
  \item $|Q| \leq X$, and
  \item $Q$ is a product of distinct primary primes of degree two, i.e. $Q$ is of the form $Q = \prod(-q)$ for distinct rational primes $q \equiv 3 \pmod{4}$.
  \end{itemize}
  Fix an element $d \in O$ of the form
  \begin{equation*}
    d = \prod_{i \in I} \pi_i \cdot
    \prod_{j \in J} (-q_j)^{n_j},
  \end{equation*}
  where $n_j$ are positive integers, $\pi_i$ are primary primes of degree one, $-q_j$ are
  primary primes of degree two, and $I$, $J$ are finite sets.  For such
  a $d$, let us consider the elliptic curve $E_d$ defined as in
   \eqref{eqn:intro} together with 
  their twists $E_{dQ}$ and their
  corresponding Hecke characters $\chi_{dQ}$ for $Q\in\mathcal{Q}(X)$. Then,
  \begin{equation*}
    \lim_{X \to \infty}\frac{1}{|\cQ(X)|} \sum_{Q \in \mathcal{Q}(X)} \frac{w(\chi_{dQ})}{w_2(\chi_{dQ})} = 0.
  \end{equation*}
\end{theorem}

The ``ideal'' statement in this direction would be
\begin{equation*}
\lim_{X \to \infty} \frac{1}{|\cE_d(X)|} \sum_{E \in \cE_d(X)}w(\chi_E) = 0,
\end{equation*}
where $\cE_d(X)$ is the set of elliptic curves obtained by twisting
$E_d$ with an element in $O$ whose absolute value is less than $X$.
However, we have a trouble in computing the local root number at a
prime above $2$, and we only have a symmetry on primes of degree two.
For these reasons, we use $\mathcal{Q}(X)$ instead of $\mathcal{E}_d$
and  $w(\chi_E)/w_2(\chi_E)$ instead of $w(\chi_E)$.


This paper is organised as follows: in Section
\ref{sec:preliminaries}, we recall some basic facts about the Gaussian
integers $O$, elliptic curves with complex multiplication by $O$, and
Hecke characters associated by elliptic curves with complex
multiplication by $O$.  In Section \ref{sec:gener-root-numb} we give
certain concrete computations on such Hecke characters, even though
these topics are already studied by \cite{RS}. In Section
\ref{sec:proofs-main-results} we will give a proof for the main
results of this paper.

\section{Preliminaries}
\label{sec:preliminaries}

\subsection{Gaussian primes}
\label{sec:gaussian-primes}

Consider the number field $K = \mathbf{Q}(i)$ and its ring of
integers $O = \mathbf{Z}[i]$.  The elements of
$O$ are called \emph{Gaussian integers}.  The unit elements
amongst them are exactly $1, -1, i$ and $-i$.  A
non-unit $\alpha \in O$ is called to be \emph{primary} if
$\alpha \equiv 1 \pmod{2(1+i)}$.  Such primary elements
$\alpha = a + bi$ are classified as the elements with
$(a,b) \equiv (1,0) \pmod{4}$ or $(a,b) \equiv (3,2) \pmod{4}$.  It is
well-known that if $\alpha$ is a non-unit element in $O$
with $(1+i) \nmid \alpha$, there is a unique unit $u$ such
that $\tilde{\alpha} = u\alpha$ is primary.  Since $K$ has class
number one, the ring $O$ is a unique factorisation domain,
and thus for any non-zero $d \in O$, we have a unique primary
prime decomposition
\begin{equation}
  \label{eq:primary-prime-decomposition}
  d = i^{n_u} \cdot (1 + i)^{n_2} \cdot \prod_{\pi:
    \text{primary primes}} \pi^{n_{\pi}},
\end{equation}
with $n_u \in \{ 0, 1, 2, 3 \}$ and $n_2, n_{\pi} \ge 0$ being
integers, where the product runs over all \emph{odd} (elements not
divisible by $(1 + i)$) primary prime elements in $O$.  We say that a
prime ideal of $O$, or a prime element generating the ideal, or even
the valuation corresponding the ideal, is \emph{of degree one}
(resp.~\emph{of degree two}) if its residual degree is equal to one
(resp.~two).


\subsection{Quartic residue symbol and quartic reciprocity}
\label{sec:quart-resid-symb}

Suppose that $\pi$ is an odd prime element of $O$ and $\alpha \in O$
is an element prime to $\pi$.  The \emph{quartic residue character} of
$\alpha \pmod{\pi}$ is the unique unit $i^k$
($k \in \{ 0, 1, 2, 3 \}$) such that
$\alpha^{\frac{N\pi - 1}{4}} \equiv i^k \pmod{\pi}$.  Here $N\pi$ is
the absolute norm of $\pi$ (i.e.~the size of the residue field
$O/\pi O$) and the existence of such unit is guaranteed by Fermat's
little theorem: $\alpha^{N\pi -1} \equiv 1 \pmod{\pi}$.  The quartic
residue character of $\alpha$ modulo $\pi$ is denoted by
$\quasym{\alpha}{\pi}$.  We use the following well-known properties on
the quartic residue character in the present paper.

\begin{proposition}[Properties of the quartic residue symbol]
\label{prop:QRS-properties}
Let $\pi = a + bi$ be an odd prime element in $O$
and $\alpha, \beta \in O$ Gaussian integers prime to $\pi$.
Then the following hold.
  \begin{enumerate}
  \item $\quasym{\alpha\beta}{\pi} = \quasym{\alpha}{\pi}
    \quasym{\beta}{\pi}$.
  \item $\overline{\quasym{\alpha}{\pi}} =
    \quasym{\overline{\alpha}}{\overline{\pi}}$, where
    $\overline{(\cdot)}$ denotes the complex conjugation.
  \item If $\pi'$ is an associate of $\pi$, then $\quasym{\alpha}{\pi}
    = \quasym{\alpha}{\pi'}$.
  \item If $\alpha \equiv \beta \pmod{\pi}$, then
    $\quasym{\alpha}{\pi} = \quasym{\beta}{\pi}$.
  \item Let $q$ be a rational prime congruent to 3 modulo 4.  Then for
    any $a \in \mathbf{Z}$ with $a \neq 0$ and $q \nmid a$, we have
    $\quasym{a}{q} = 1$.
  \end{enumerate}
\end{proposition}

For future reference, here we also give the quartic reciprocity law.

\begin{theorem}[Quartic Reciprocity]
  \label{th:quartic-reciprocity}
  Suppose that $\alpha, \beta \in O$ are primary (hence not
  even) relatively prime non-unit elements.  Then
  \begin{equation}
    \label{eq:main-reciprocity}
    \quasym{\alpha}{\beta} \quasym{\beta}{\alpha}^{-1} =
    (-1)^{\frac{N\alpha - 1}{4} \frac{N\beta - 1}{4}},
  \end{equation}
  or equivalently,
  \begin{equation}
    \label{eq:main-reciprocity-2}
    \quasym{\beta}{\alpha} =
    \begin{cases}
      - \quasym{\alpha}{\beta} & \text{ if both $\alpha$ and $\beta$
        are $\equiv 3+ 2i \pmod{4}$,}\\
      \quasym{\alpha}{\beta} & \text{ otherwise.}
    \end{cases}
  \end{equation}
  Furthermore, we have the following auxiliary relations for a primary
  prime $\pi = a + bi$:
  \begin{equation}
    \label{eq:reciprocity-units-and-evens}
    \quasym{i}{\pi} = i^{\frac{-a+1}{2}} \qquad
    \text{and} \qquad \quasym{1 + i}{\pi} = i^{\frac{a - 1 -b -b^2}{4}}.
  \end{equation}
\end{theorem}

\subsection{Elliptic curves having CM by
  \texorpdfstring{$O$}{the Gaussian ring}}
\label{sec:cm-elliptic-curves}

Let $E$ be an elliptic curve defined over the algebraic closure of
$\mathbf{Q}$ having complex multiplication by $O$.  From the
classical theory of complex multiplication, such $E$ is defined over
the Hilbert class field of $K$, which is $K$ itself and its
$j$-invariant is $1728$.  Furthermore, such elliptic curves are
defined by the Weierstrass equation
\begin{equation}
  \label{eq:Weq}
  y^2 = x^3 - dx,\qquad d \in O
\end{equation}
where the discriminant of the equation $\Delta = 2^6d^3 \neq 0$.  The
curve defined by the equation \eqref{eq:Weq} is denoted by $E_d$
 over $K$.
Note that two such equations for $d$ and $d'$ define the 
$K$-isomorphic
curve if and only if $d' = x^4d$ for some $x \in O$.  For convenience,
we therefore assume $d$ is fourth-power-free.

For future reference, we also investigate the reduction type of $E$
modulo the prime $(1+i)$.  This can be easily verified by
using Tate's algorithm (e.g.~see \cite[\S IV.9]{advAEC}).

\begin{lemma}
  \label{Tate}
  Let $E/K$ be an elliptic curve defined by a Weierstrass equation
  \eqref{eq:Weq} with fourth-power-free $d$.  We write
  $\pi = 1 + i \in K$ and consider the power series expansion of $d$
  in the completion $K_{\pi}$ of $K$ with respect to $\pi$:
  \begin{equation}
    \label{eq:expansion-of-d}
    d = \sum_{j=0}^{\infty} d_j \pi^j, \qquad d_j \in \{ 0, 1 \}.
  \end{equation}
  Then the reduction types of $E$ and corresponding conductor
  exponents are given in terms of the values of $d_j$ in the
  Table \ref{table:1}.

  \begin{table}[ht]
\begin{tabular}{@{}llllllll@{}}
\toprule
$d_0$ & $d_1$ & $d_2$ & $d_3$ & $d_4$ & $d_5$ & Reduction Type      & Conductor Exponent \\ \midrule
0     & 0     & 0     & 1     & -     & -     & $\mathrm{III}^\ast$ & 14                 \\
0     & 0     & 1     & 0     & -     & -     & $\mathrm{I}_4^\ast$ & 10                 \\
0     & 0     & 1     & 1     & -     & -     & $\mathrm{I}_2^\ast$ & 12                 \\
0     & 1     & -     & -     & -     & -     & $\mathrm{III}$      & 14                 \\
1     & 0     & 0     & 0     & -     & -     & $\mathrm{I}_2^\ast$ & 6                  \\
1     & 0     & 0     & 1     & -     & -     & $\mathrm{I}_0^\ast$ & 8                  \\
1     & 0     & 1     & 0     & 0     & 0     & good                & 0                  \\
1     & 0     & 1     & 0     & 0     & 1     & $\mathrm{II}^\ast$  & 4                  \\
1     & 0     & 1     & 0     & 1     & 0     & $\mathrm{II}^\ast$  & 4                  \\
1     & 0     & 1     & 0     & 1     & 1     & good                & 0                  \\
1     & 0     & 1     & 1     & -     & -     & $\mathrm{I}_0^\ast$ & 8                  \\
1     & 1     & -     & -     & -     & -     & $\mathrm{II}$       & 12                 \\ \bottomrule
\end{tabular}
\caption{Reduction types and conductor exponents of the elliptic curve
  $E_d$, in terms of the power series expansion of $d$
  (cf.~\eqref{eq:expansion-of-d}).  Here - means that the
  corresponding value is irrelevant.}
\label{table:1}
\end{table}
\end{lemma}

\subsection{Hecke characters associated with CM elliptic curves and
  their root numbers}
\label{sec:hecke-char-assoc}

For a non-zero ideal $\mathfrak{f}$ of $O$, we define:
\begin{align*}
  K(\mathfrak{f}) &= \{ \alpha \in K^{\times} : (\alpha O,
                    \mathfrak{f}) = 1 \}, \\
  K_{\mathfrak{f}} &= 1 + K(\mathfrak{f})\mathfrak{f} = \{ \alpha \in
                     K^{\times} : \alpha \equiv 1
                     \operatorname{mod}^{\times} \mathfrak{f} \}
                     \subset K(\mathfrak{f}), \\
  I(\mathfrak{f}) &= \{ \text{fractional ideals of $K$ coprime to
                    $\mathfrak{f}$} \}, \\
  P(\mathfrak{f}) &= \{ \text{principal fractional ideals
                    $\alpha O$ of $K$ coprime to
                    $\mathfrak{f}$} \}, \\
  P_{\mathfrak{f}} &= \{ \text{principal fractional ideals
                     $\alpha O$ of $K$ where $\alpha \equiv 1
                     \operatorname{mod}^{\times} \mathfrak{f}$ }\}.
\end{align*}

Let $E = E_d$ be an elliptic curve defined over $K$ with the same
assumptions as in Subsection \ref{sec:cm-elliptic-curves}.  From the
theory of complex multiplication of elliptic curves, there is a Hecke
character
$\chi_{\mathbf{A}}: \mathbf{A}_K^{\times}/K^{\times} \to \mathbf{C}^{\times}$
where $\mathbf{A}^{\times}_K$ is the idèle group of $K$.  In terms of
ideal groups, this can be viewed as a pair $(\chi, \chi_{\infty})$,
where $\chi: I(\mathfrak{f}) \to \mathbf{C}^{\times}$ is a character
from the ideal group co-prime to a non-zero ideal $\mathfrak{f}$ of $O$,
and
$\chi_{\infty}:(\mathbf{R} \otimes_{\mathbf{Q}} K)^{\times}=
\mathbf{C}^{\times} \to \mathbf{C}^{\times}$ is a continuous
character.  Here the two characters $\chi$ and $\chi_{\infty}$ are
related in the following way: for $\alpha \in K_{\mathfrak{f}}$, we
have $\chi(\alpha O) = \chi_{\infty}^{-1} (1 \otimes \alpha)$.
Furthermore, we have more refined relation giving the values of the
character at $\alpha \in K(\mathfrak{f})$: we have a character
$\epsilon: (O/\mathfrak{f})^{\times} \cong
K(\mathfrak{f})/K_{\mathfrak{f}} \to S^1 \subset \mathbf{C}^{\times}$
such that
$\chi(\alpha O) = \epsilon(\alpha K_{\mathfrak{f}})
\chi_{\infty}^{-1}(1 \otimes \alpha)$ for each
$\alpha \in K(\mathfrak{f})$.  Such a Hecke character
$\chi$ is now called a \emph{Hecke character with conductor
  $\mathfrak{f}$, $(O/\mathfrak{f})^{\times}$-type $\epsilon$ and
  infinity-type $\chi_{\infty}$}.  The non-zero ideal $\mathfrak{f}$ is
called the \emph{conductor} of the Hecke character $\chi$ and it is
known that it has the same prime factors as the conductor of the curve
$E$.

\section{Generalities of root numbers of Hecke characters associated
  with CM elliptic curves}
\label{sec:gener-root-numb}

\subsection{Explicit computation of Hecke characters}
\label{sec:expl-comp-hecke}

From \S\ref{sec:preliminaries}, we see that there is an element
$d \in O$ unique up to multiplication by a fourth-power in $O$ such
that $E$ is isomorphic over $K$
to $E_d$ (cf.~Equation \eqref{eq:Weq}).  We
choose $d$ itself to be fourth-power-free, i.e.~in the unique primary
prime decomposition \eqref{eq:primary-prime-decomposition} of $d$, we
have $0 \le n_u, n_2, n_{\pi} \le 3$ for each odd prime $\pi$ of $K$.
This means that Equation \eqref{eq:Weq} defining $E_d$ is minimal at
each prime except $(1+ i)O$.  In particular, the curve
$E_d$ has good reduction modulo any \emph{odd} prime not dividing $d$.

Let $\chi: I(\mathfrak{f}) \to \mathbf{C}^{\times}$ be the Hecke
character associated with the curve $E_d$.  Here $\mathfrak{f}$ is an
integral ideal of $O$ having the same prime ideal divisors as the
conductor of the curve $E_d$.

\begin{lemma}\label{lem:point-counting}
  Let $\mathfrak{p}$ be a prime ideal of $K$ such that $\mathfrak{p}
  \nmid 2\mathfrak{f}$ with residue field $k = k_{\mathfrak{p}}$.  If
  $E = E_d$ is the elliptic curve defined by the Weierstrass equation
  \eqref{eq:Weq}, then concerning the number of $k$-rational points on
  the reduction $\tilde{E}$ of $E$ modulo $\mathfrak{p}$, we have the
  following results.
  \begin{itemize}
  \item If $\mathfrak{p}$ is of degree one, i.e.~$k = \mathbf{F}_p$
    for some rational prime $p \equiv 1 \pmod{4}$, we have
    \begin{equation*}
      \# \tilde{E}(k) = p + 1 - \overline{\quasym{d}{\pi}} \pi -
      \quasym{d}{\pi} \overline{\pi},
    \end{equation*}
    where $p = \pi \overline{\pi}$ is the decomposition of $p$ in $K$
    with primary prime elements $\pi$ and $\overline{\pi}$.
  \item If $\mathfrak{p}$ is of degree two, i.e.~$k$ is a degree two
    extension of $\mathbf{F}_p$ for some rational prime
    $p \equiv 3 \pmod{4}$, we have
    \begin{equation*}
      \# \tilde{E}(k) =
      \begin{cases}
        p^2 + 1 & \text{$d$ is not a square in $k$,} \\
        p^2 + 1 - 2p & \text{$d$ is a square but not a fourth power in
          $k$,} \\
        p^2 + 1 + 2p & \text{$d$ is a fourth power in $k$.}
      \end{cases}
    \end{equation*}
  \end{itemize}
\end{lemma}

\begin{proof}
  Over the rationals $\mathbf{Q}$, it is proved in \cite[\S 18.4,
  Theorem 5]{IR}.  We can prove the result similarly over $K$.
\end{proof}

\begin{proposition}
  \label{prop:char-value}
  The Hecke character $\chi: I(\mathfrak{f}) \to \mathbf{C}^{\times}$
  associated with the curve $E_d$ defined by the equation
  \eqref{eq:Weq} is given by, for any prime ideal $\mathfrak{p}$ of
  $K$ co-prime to $2\mathfrak{f}$,
  \begin{equation*}
    \chi(\mathfrak{p}) = \overline{\quasym{d}{\mathfrak{p}}} \pi,
  \end{equation*}
  where $\pi$ is the unique primary generator of $\mathfrak{p}$.
\end{proposition}

\begin{proof}
  This is due to Lemma \ref{lem:point-counting}, as explained in
  \cite[Example II.10.6]{advAEC}.
\end{proof}

From this, we can identify the infinity type of the Hecke character
$\chi_{\infty}$.

\begin{lemma}
  Let $E = E_d$ be an elliptic curve over $K$ with the same
  assumptions as in Subsection \ref{sec:cm-elliptic-curves}.  Then the
  infinity type
  $\chi_{\infty}: \mathbf{C}^{\times} \to \mathbf{C}^{\times}$ of the
  corresponding Hecke character $\chi$ is given by
  $z \mapsto z^{-1}$.
\end{lemma}

\begin{proof}
  Since the Hecke characters attached to CM elliptic curves are of
  weight 1 (cf.~\cite[\S 1.1]{Schap}), the infinity type must be
  either of the form $z \mapsto z^{-1}$ or of the form
  $z \mapsto \bar{z}^{-1}$.  In order to determine which one is the
  true infinity type for our case, we consider an element
  $\alpha \in O$ such that $\alpha \equiv 1 \pmod{\mathfrak{p}^n}$
  where $\mathfrak{p}$ is an odd prime with
  $\mathfrak{p}^n  \| \mathfrak{f}$, $n > 0$, that
  $\alpha \equiv 1 \pmod{(1+i)^m O}$ where $m$ is any integer
  $\ge \max(\ord_{(1+i)O} \mathfrak{f}, 3)$ and that $\alpha O$ is a
  prime ideal of $O$.  The existence of such an element is guaranteed
  by Lemma \ref{Dirichlet}.  Now, since
  $\alpha O \in P_{\mathfrak{f}}$ and since
  $\overline{(d/\alpha O)_4} = 1$ due to the conditions of $\alpha$,
  we see that $\chi_{\infty}(\alpha) = \alpha^{-1}$.  This implies
  that the infinity type is indeed $z \mapsto z^{-1}$.
\end{proof}

\subsection{Root numbers of Hecke characters}
\label{sec:root-numbers-hecke}

Let $E$ be an elliptic curve defined over $K$ given by the Weierstrass
equation \eqref{eq:Weq} with fourth-power-free $d \in O$.  The
\emph{(global) root number} of the associated Hecke character $\chi$
of $E$ can be represented by the product of local root numbers of
$\chi$.  In our case, for any place $v$ of $K$, we can compute the
local root number $w_v(\chi)$ as follows with three cases
(cf.~\cite[Lecture 2]{Rohr}, or \cite[\S\S 3--5]{Watkins}):
\begin{enumerate}
\item If $v$ is archimedean, in other words, since our $K$ is
  imaginary quadratic, if $v$ is the unique complex place of $K$, then
  $w_v(\chi) = -i$.
\item If $\chi$ is unramified at $v$, in other words, if $E$ has good
  reduction at $v$, then
  $w_v(\chi) = \chi^v(\mathfrak{d}_v) / |\chi^v(\mathfrak{d}_v)|$,
  where $\chi^v$ is the local component of $\chi_{\mathbf{A}}$,
  i.e.~$\chi^v = \chi_{\mathbf{A}} \circ \iota_v$, where
  $\iota_v: K_v^{\times} \to \mathbf{A}_K^{\times}/K^{\times}$ is
  induced by the natural inclusion and $\mathfrak{d}_v$ is the local
  different of the completion $K_v$ of $K$ with respect to $v$.  In
  particular, if moreover $v$ is odd, then $w_v(\chi) = +1$.  On the
  other hand, if $v$ is the unique place lying over the prime $2$,
  i.e.~the place corresponds to the prime ideal $(1+i)O$ of $K$ and if
  $E$ has good reduction at $v$, then
  $w_v(\chi) = \chi(\mathfrak{d}_v) = \chi((1+i)^3O)$.
\item Suppose $\chi$ is ramified at $v$, i.e.~$E$ has bad reduction at
  $v$.  Let $\chi^v_u = \chi^v / |\chi^v|$.  we have
  \begin{equation}
    \label{formula}
    w_v(\chi) = \chi_u^v(\beta) \cdot G(\chi^v),
  \end{equation}
  where $\beta \in K^{\times}$ is any element satisfying
  $v(\beta) = a(\chi^v) + n$.  Here $a(\chi^v)$ is the exponent of the
  conductor of $\chi^v$ (cf.~\cite{Rohr}, pp.~28--29), and $n$ is the
  valuation of the local different ideal of $K_v$.  Moreover,
  $G(\chi^v)$ is the normalised Gauss sum given by:
  \begin{equation*}
    G(\chi^v) = q^{-a(\chi^v)/2} \cdot \sum_{x \in
      (O_v/\mathfrak{f}(\chi^{v}))^{\times}} (\chi^v)^{-1}(x) \cdot e^{2\pi i
      \operatorname{tr}^{K_v}_{\mathbf{Q}_p} (x/\beta)},
  \end{equation*}
   where $q$ is the cardinality of the residue field $k_{v}$.
\end{enumerate}

\begin{remark}
  \begin{enumerate}
  \item The local character $\chi^v: K_v^{\times} \to
    \mathbf{C}^{\times}$ is ramified if and only if the original curve
    $E$ has bad reduction at $v$.  Moreover, the exponent of the conductor
    $a(\chi^v)$ of $\chi^v$ is exactly the half of the exponent of the
    conductor of $E$ at $v$ (cf.~\cite[Theorem 12]{ST}).  In
    particular, if $v$ is odd and $E$ has bad reduction at $v$, then
    $a(\chi^v) = 1$.
  \item We can compute the values of $\chi^v$ using \cite[Proposition
    2.1]{Rohr}.  More precisely, suppose that $\mathfrak{p}$ is a
    prime ideal of $K$ with corresponding finite place $v$ and $\pi$
    is a uniformiser for $v$.  Then we have the following.
    \begin{itemize}
    \item If $\mathfrak{p} \nmid \mathfrak{f}$, then $\chi^v(\pi) =
      \chi(\mathfrak{p})$.
    \item If $\mathfrak{p} \mid \mathfrak{f}$ then $\chi^v|O_v^{\times}
      = \epsilon_v^{-1}$.
    \item If $\beta \in O_K$ and $\beta O_K$ is a power of some prime
      ideal $\mathfrak{p}$ dividing $\mathfrak{f}$, then
      \begin{equation*}
        \chi^v(\beta) = \chi_{\infty}^{-1}(\beta) \cdot
        \prod_{\mathfrak{q} \mid \mathfrak{f},\,\,\mathfrak{q} \neq
          \mathfrak{p}} \epsilon_{\mathfrak{q}}(\beta).
      \end{equation*}
    \end{itemize}
  \end{enumerate}
\end{remark}

\section{Proofs of the main results}
\label{sec:proofs-main-results}

The following result is a generalisation of Dirichlet's theorem on
arithmetic progression over the Gaussian field $K$.

\begin{lemma} \label{Dirichlet} Let
  $\mathfrak{m}_1, \cdots, \mathfrak{m}_J$ be mutually relatively
  prime ideals of $O$ and $\alpha_j \in O$ for $j = 1, \cdots, J$ be
  given.  Then, there are infinitely many prime elements $x \in O$
  such that $x \equiv \alpha_j \pmod{\mathfrak{m}_j}$ for all
  $j = 1, \cdots, J$.
\end{lemma}


Let us consider the elliptic curve $E_d$ defined by the equation
\eqref{eq:Weq} with fourth-power-free $d \in O$ and the primary prime
decomposition (cf.~Subsection \ref{sec:gaussian-primes}) of $d$:
\begin{equation}
  \label{eq:d-expression}
  d = i^{n_u} \cdot (1 + i)^{n_2} \cdot \prod_{\pi:
    \text{degree 1}} \pi^{n_{\pi}} \cdot \prod_{-q: \text{degree 2}}
  (-q)^{n_q}.
\end{equation}
Here, we have separated prime elements of $K$ into two kinds: those
with absolute residual degree one or two.  Note also that the primary
generator of a prime of degree two is of the form $-q$ with $q$ being
a rational prime congruent to $3$ modulo $4$.  For primes of degree
one, we denote by $\pi$ its primary generator.  Also note the range of
the exponents $n_{\bullet}$ of the factors in \eqref{eq:d-expression}:
we have $0 \le n_{\bullet} < 4$ since $i^4 = 1$ and $d$ is
assumed to be fourth-power-free.


\begin{lemma}
  \label{lem:epsilon-determined}
  Suppose that $v$ is a finite place of $K$ with corresponding prime
  ideal $\mathfrak{p}_v$ of $O$, and assume that
  $\mathfrak{p}_v \nmid 2$ and that $\mathfrak{p}_v$ is of degree one.
  Let $d$ be a fourth-power-free element in $O$.  Then the local
  epsilon type $\epsilon_v$ associated with $E_d$ at $v$ is completely
  determined by the value $v(d)$.  More precisely, $\epsilon$ is of
  exact order 4 (resp.~2, resp.~1) if and only if
  $v(d) \in \{ 1, 3 \}$ (resp.~$v(d) = 2$, resp.~$v(d) = 0$).
\end{lemma}



\begin{proof}
  Note that the residue field of $\mathfrak{p}_v$ is $\mathbf{F}_p$
  for a rational prime $p \equiv 1 \pmod{4}$.  Fix an element $g$ of
  $O$ which generates the multiplicative group $\mathbf{F}_p^{\times}$
  modulo $\mathfrak{p}_v$.  By Lemma \ref{Dirichlet}, we can find a
  primary prime element $x \in O$ and $u \in O^{\times}$ such that
  \begin{equation*}
    u x \equiv
    \begin{cases}
      g & \text{modulo $\mathfrak{p}_v$}, \\
      1 & \text{modulo $\mathfrak{p}_w$ for all $w$ with $w(d) > 0$,
        $w(2) = 0$ and $w \neq v$}, \\
      1 & \text{modulo $16O$}.
    \end{cases}
  \end{equation*}
  Note that $u = 1$ because both $x$ and $ux$ are primary.  Write
  $x = a+ bi$ with $a, b \in \mathbf{Z}$, $a \equiv 1 \pmod{16}$ and
  $b \equiv 0 \pmod{16}$.  Expand $d$ as the product of primary
  elements,
  i.e.~$d = i^{n_u} \cdot (1+i)^{n_2} \cdot \prod_{\mathfrak{l} \mid
    d, \text{ odd}} \pi_{\mathfrak{l}}^{n_{\mathfrak{l}}}$, with
  $\pi_{\mathfrak{l}}$ being primary odd primes.  Then,
  \begin{eqnarray*}
    \epsilon_v(g) &=& \epsilon(x) = \chi(xO)\chi_{\infty}(x) \\
                &=&
                    \overline{\quasym{i}{x}^{n_u}} \cdot
                    \overline{\quasym{1+i}{x}^{n_2}} \cdot \prod_{\mathfrak{l} \mid d,
                    \text{ odd}}
                    \overline{\quasym{\pi_{\mathfrak{l}}}{x}^{n_{\mathfrak{l}}}}\\
                &=&
                    \overline{i^{-\frac{a-1}{2} n_u}} \cdot
                    \overline{i^{\frac{a-b-1-b^2}{4}n_2}} \cdot
                    \overline{\quasym{g}{\mathfrak{p}_v}^{n_{\mathfrak{p}_v}}} =
                    \overline{\quasym{g}{\mathfrak{p}_v}^{n_{\mathfrak{p}_v}}},
  \end{eqnarray*}
  by Proposition \ref{prop:char-value} and quartic reciprocity
  (Theorem \ref{th:quartic-reciprocity}).  As
  $(\cdot/\mathfrak{p}_v)_4$ is of exact order 4, so $\epsilon$ is of
  exact order 4 (resp.~2, resp.~1) when
  $n_{\mathfrak{p}_v} = v(d) \in \{ 1, 3 \}$ (resp.~$= 2$,
  resp.~$= 0$).
\end{proof}

\begin{proposition}
  \label{prop:determination-of-root-number}
  Let $d \in O$ be a fourth-power-free element and $E_d$ the elliptic
  curve defined by \eqref{eq:Weq} with $d$ and with corresponding
  Hecke character $\chi$.
  \begin{enumerate}
  \item Suppose that $\pi$ arbitrary primary prime of degree one
    dividing $d$ with order $n \in \{ 1, 2, 3 \}$, and $v$ its
    corresponding valuation of $K$.  Then,
    \begin{equation}\label{eq:root-number-degree-one}
      w_{\pi}(\chi) = \eta^n \cdot \frac{\pi}{|\pi|} \cdot
      \overline{\quasym{d / \pi^n}{\pi}} \cdot
      \overline{\quasym{\overline{\pi}^{-1}}{\pi}^n} \cdot G(\chi^v),
    \end{equation}
    where $\eta \in \{ \pm 1 \}$.  Here $\eta = -1$ if and only if
    $\pi \equiv 3 + 2i \pmod{4}$.
  \item Let $q$ be a rational prime dividing $d$ congruent to 3 modulo
    4 and $v$ the valuation corresponding to the prime $qO$.  Then,
    \begin{equation}
      \label{eq:root-number-degree-two}
      w_q(\chi) = - \overline{\quasym{d/(-q)^n}{-q}} \cdot G(\chi^v).
    \end{equation}
    In this case, one has
    \begin{equation*}
      G(\chi^v) =
      \begin{cases}
        1 & \text{ when $\chi^v$ has exact order 2,} \\
        -1 & \text{ when $\chi^v$ has exact order 4.}
      \end{cases}
    \end{equation*}
  \end{enumerate}
\end{proposition}

\begin{proof}
  (1) Let $\mathfrak{f}$ be the conductor of the character $\chi$.  We
  compute $w_{\pi}(\chi)$ following the formula \eqref{formula} with
  picking $\beta = p$ as follows.
  \begin{equation*}
    w_{\pi}(\chi) = \frac{\chi^v(p)}{|\chi^v(p)|} \cdot
    G(\chi^v) = \frac{\pi}{|\pi|} \cdot \chi^v(\overline{\pi}) \cdot
    \prod_{w(\mathfrak{f}) > 0,\,w \neq v} \epsilon_w(\pi) \cdot G(\chi^v).
  \end{equation*}
  Let $\mathfrak{p}_w$ be the prime ideal corresponding to each
  valuation $w \neq v$ such that $w(\mathfrak{f}) > 0$,
  i.e.~$\mathfrak{p}_w$'s are prime ideal factors of $\mathfrak{f}$.
  In order to compute the product of local epsilon factors in the
  above expression, we choose, by Lemma \ref{Dirichlet}, a primary
  prime element $x \in O$ and a unit $u \in O^{\times}$ so that
  \begin{equation*}
    ux \equiv
    \begin{cases}
      \pi \pmod{(1+i)^8} \\
      \pi \pmod{\mathfrak{p}_w} & \text{ for odd
        $w \neq v$ and $w(\mathfrak{f}) > 0$,} \\
      \overline{\pi}^{-1} \pmod{\mathfrak{p}_v^n}.
    \end{cases}
  \end{equation*}
  Since $\pi$ is primary, the first condition ensures that $u = 1$,
  and all of the conditions above are arranged to imply
  \begin{equation*}
    \chi^v(\overline{\pi}) \cdot \prod_{w(\mathfrak{f}) > 0, \, w \neq
      v} \epsilon_w(\pi) = \epsilon(x) = \overline{\quasym{d}{x}}.
  \end{equation*}
  Therefore,
  \begin{multline*}
    w_v(\chi) = \frac{\pi}{|\pi|} \cdot \epsilon(x) \cdot G(\chi^v) =
    \frac{\pi}{|\pi|} \cdot \overline{\quasym{d/\pi^n}{x}} \cdot
    \overline{\quasym{\pi}{x}^n} \cdot G(\chi^v) \\
    = \frac{\pi}{|\pi|}
    \cdot \overline{\quasym{i^{n_u}}{x}} \cdot
    \overline{\quasym{(1+i)^{n_2}}{x}} \cdot
    \overline{\quasym{d/(i^{n_u}(1+i)^{n_2}\pi^n)}{x}} \cdot
    \overline{\quasym{\pi}{x}^n} \cdot G(\chi^v).
  \end{multline*}
  Now one can observe the following.
  \begin{itemize}
  \item It follows from Equation
    \eqref{eq:reciprocity-units-and-evens} that $(i/x)_4 = (i/\pi)_4$
    and $((1+i)/x)_4 = ((1+i)/\pi)_4$.
  \item Since ${d}/(i^{n_u}(1+i)^{n_2}\pi^n)$ is the product of
    primary prime elements, by factoring it into primary primes and
    interchanging the ``denominators'' and the ``numerators'' in the
    quartic residue symbol by Equation \eqref{eq:main-reciprocity-2},
    we get
    \begin{equation*}
      \quasym{d/(i^{n_u}(1+i)^{n_2}\pi^n)}{x} = \quasym{d/(i^{n_u}(1+i)^{n_2}\pi^n)}{\pi}.
    \end{equation*}
  \item Since $x \equiv \pi \pmod{4}$, we see $(\pi/x)_4 = \eta \cdot
    (x/\pi)_4 = \eta \cdot (\overline{\pi}^{-1}/ \pi)_4$.
  \end{itemize}
  Summarising, we obtain Equation \eqref{eq:root-number-degree-one}.

  (2) In this case, the normalised Gauss sum $G(\chi^v) = 1$
  (resp.~$-1$) if and only if $\chi^v$ has exact order $2$ (resp.~$4$)
  (cf.~\cite{MV}, p.~383 and \cite[Theorem 2.4]{Mbodj}).  This shows
  the last statement.  Now, by Formula \eqref{formula} with
  $\beta = q$,
  \begin{equation*}
    w_{\pi}(\chi) = \frac{\chi^v(q)}{|\chi^v(q)|} \cdot
    G(\chi^v) = \prod_{w(\mathfrak{f}) > 0,\,w \neq v} \epsilon_w(q) \cdot G(\chi^v).
  \end{equation*}
  Let $\mathfrak{p}_w$ be the prime ideal corresponding to each
  valuation $w \neq v$ such that $w(\mathfrak{f}) > 0$,
  i.e.~$\mathfrak{p}_w$'s are prime ideal factors of $\mathfrak{f}$.
  As above, we choose a primary prime element $x \in O$ and a unit
  $u \in O^{\times}$ so that
  \begin{equation*}
    ux \equiv
    \begin{cases}
      q \pmod{(1+i)^{\max(\ord_{(1+i)} \mathfrak{f},\, 4)}} \\
      q \pmod{\mathfrak{p}_w} & \text{ for odd
        $w \neq v$ and $w(\mathfrak{f}) > 0$,} \\
      1 \pmod{\mathfrak{p}_v^n}.
    \end{cases}
  \end{equation*}
  Since $-q$ is primary, the first condition ensures that $u = -1$,
  and all of the conditions above are arranged to imply
  \begin{equation*}
    w_q(\chi) = \epsilon(-x) \cdot G(\chi^v) = -
    \overline{\quasym{d}{x}} G(\chi^v) = -
    \overline{\quasym{d/(-q)^n}{x}} \cdot
    \overline{\quasym{(-q)^n}{x}} \cdot G(\chi^v).
  \end{equation*}
  Now, in the same fashion as above, we can see
  \begin{equation*}
    \quasym{d/(-q)^n}{x} = \quasym{d/(-q)^n}{q}, \quad \text{and}
    \quad \quasym{(-q)^n}{x} = \quasym{-1}{-q}^n = 1,
  \end{equation*}
  whence \eqref{eq:root-number-degree-two}.
\end{proof}

\begin{corollary}
  \label{cor:nusym}
  Let $\pi$ be a primary prime element of degree one of $O$ such that
  $\pi \equiv 3 + 2i \pmod{4}$ and $v$ its corresponding valuation of
  $K$.  Consider the elliptic curve $E_{\pi}$ defined by
  \eqref{eq:Weq} with $d = \pi$ and the corresponding Hecke character
  $\chi_{\pi}$.  Then the set
  $\{ w_v(\chi_d) / w_v(\chi_{\pi}) :  d = i^m
  \pi \text{ with } m = 0, 1, 2, 3 \}$ is
  identical to the set of the fourth roots of unity, $\mu_4$.
\end{corollary}

\begin{proof}
  The defining equation for $E$ is also given by Equation
  \eqref{eq:Weq} with $d = i^m \pi$ for some $m = 0, 1, 2, 3$.  Then
  it follows from Lemma \ref{lem:epsilon-determined} and Proposition
  \ref{prop:determination-of-root-number} (1), we see that
  $w_v(\chi_d)/w_v(\chi_{\pi}) = \overline{\quasym{i}{\pi}^m}$.  Since
  $\pi \equiv 3 + 2i \pmod{4}$, by Formula
  \eqref{eq:reciprocity-units-and-evens} we have $\quasym{i}{\pi} =
  \pm i$, whence the result follows.
\end{proof}

Now we are ready to prove Theorem \ref{main nearball}.

\begin{proof}[Proof of Theorem \ref{main nearball}]
  Let $E$ be the elliptic curve defined by the Weierstrass equation
  \eqref{eq:Weq} with $d = \pi$ being a primary Gaussian prime of
  degree one.  Also, let $v$ be the place of $K$ corresponding to the
  prime $\pi$ and $\chi$ the corresponding Hecke character for $E/K$.
  It follows from \cite[Theorem 2]{Matt}, that the Gauss sum
  $G(\chi^v)$ in \eqref{formula} is contained in
  $\mu_4 \cdot \sqrt{\pi}/|\sqrt{\pi}|$.  Thus,
  $w_v(\chi) \in \mu_4 \cdot \pi^{3/2}/ |\pi|^{3/2}$.  Thus, it is
  enough to show
  \begin{enumerate}
  \item that the set $\{ \pi / |\pi| : \pi \equiv 3+2i \pmod{4} \text{ being Gaussian primes of degree one}     
      \}$ is dense in the unit circle
    $S^1$, and
  \item that for each such prime $\pi$ and a fourth root $\zeta$,
    there is an elliptic curve $E$ with Hecke character $\chi$ such
    that $w_v(\chi) = \zeta\cdot \pi^{3/2}/|\pi|^{3/2}$.
  \end{enumerate}


  For (1), it can be shown in \cite{Mak}. 
  We use a corrected version of \cite[Proposition 2]{JP}.
  Let $F$ be an imaginary quadratic field, 
  and fix $\mu, \nu \in O_F$ satisfying $\mu \neq 0$ and $\nu \pmod{\mu}$ being  an invertible residue class. 
  As $x \to \infty$, we have
  \begin{equation*}
  \sum_{\substack{ \pi, \textrm{ prime in } F \\ N(\pi), \textrm{ prime in }\bQ \\
  x < N(\pi) < x + x^{0.735} \\ \pi \equiv \nu \pmod{\mu} \\ \theta_1 < \arg(\pi) < \theta_2  }} 1 \sim \frac{w_F}{h_F \phi(\mu)} \frac{\theta_2 - \theta_1}{2 \pi} \frac{x^{0.735}}{\log x},
  \end{equation*}
  when $x^{-0.265} < \theta_2 - \theta_1 < 2\pi$. Here $\phi$ is the Euler phi function for number fields, $w_F$ and  $h_F$ is the number of root of unity in $F$ and the class number of $F$.
  Hence for an arbitrary short interval $[\theta_1, \theta_2]$, there are infinitely many degree
  one primes $\pi \equiv 3+2i \pmod{4}$ such that $\theta_1 < \arg(\pi) < \theta_2$.
  


  Now let us prove (2).  By Corollary \ref{cor:nusym}, given
  $\zeta \in \mu_4$ and $\pi$, there is $m \in \{ 0, 1, 2, 3 \}$ such
  that $w_v(\chi_{i^m \pi}) = \zeta \cdot \pi^{3/2} / |\pi|^{3/2}$.
  This proves the Theorem.
\end{proof}

\begin{corollary}
  \label{cor:expanding-d}
  Let $q'$ be a rational prime congruent to 3 modulo 4, relatively
  prime to $d$, and $k \in \{1 , 2, 3\}$.
  \begin{enumerate}
  \item Let $\pi$ be a prime of degree one dividing $d$.  Then,
    \begin{equation*}
      w_\pi(\chi_{(-q')^k d}) = w_\pi(\chi_d) \cdot \overline{\quasym{-q'}{\pi}^k}.
    \end{equation*}
  \item Let $q$ be a prime of degree two dividing $d$.  Then,
    \begin{equation*}
      w_q(\chi_{(-q')^k d}) = w_q(\chi_d).
    \end{equation*}
  \end{enumerate}
\end{corollary}

\begin{proof}
  (1) From Lemma \ref{lem:epsilon-determined}, it follows that
  $G(\chi_{(-q')^k d}^\pi) = G(\chi_d^\pi)$, and hence
  \begin{equation*}
    w_\pi(\chi_{(-q')^k d}) / w_\pi(\chi_d) = \overline{\quasym{-q'}{\pi}^k},
  \end{equation*}
  by Proposition \ref{prop:determination-of-root-number}.  This gives
  the result.

  (2) The proof is the same as in (1), once one notes that
  $\quasym{-q}{-q'} = 1$ (cf.~Proposition \ref{prop:QRS-properties}
  (5)).
\end{proof}

\begin{proof}[Proof of Theorem \ref{main manytheta}]
  Given $\theta \in S^1$, let $E_d$ be the curve defined by
  \eqref{eq:Weq} as usual and $v$ a valuation of $K$ such that
  $w_v(\chi_d) = \theta$.  Suppose first that $v$ corresponds to a
  prime ideal of $O$ of degree two, i.e.~there is a rational prime
  $q \equiv 3 \pmod{4}$ such that $qO$ is the prime ideal
  corresponding to $v$.  In this case, Corollary \ref{cor:expanding-d}
  (2) shows that $w_v(\chi_{-q'd}) = w_v(\chi_d)$ for any rational
  prime $q' \neq q$ with $q' \equiv 3 \pmod{4}$.

  Suppose that $v$ corresponds to a prime ideal $\mathfrak{p} = \pi O$
  of degree one with primary prime generator $\pi \in O$.  Then by
  Lemma \ref{Dirichlet}, there are infinitely many rational primes $q$
  congruent to 3 modulo 4 such that $\quasym{-q}{\pi} = -1$, so
  $\overline{\quasym{-q}{\pi}^2} = 1$.  Then by Corollary
  \ref{cor:expanding-d} (1), for the curve $E_{(-q)^2 d}$ and for its
  corresponding Hecke character $\chi_{(-q)^2d}$, we have
  $w_v(\chi_{(-q)^2 d}) = w_v(\chi_d)$.
\end{proof}

Now we turn our focus to Theorem \ref{main average}.  We recall some classical results of Wirsing.

\begin{lemma}[{\cite[Satz 1]{Wir1}}] \label{Wirsing}
Let $f$ be a non-negative multiplicative function such that
\begin{itemize}
\item there exist constants $c_1$ and $c_2< 2$ satisfying $f(p^v) \leq c_1c_2^v$ for all primes $p$ and $v \geq 2$,
\item there exists a real number $\tau > 0$ satisfying
\begin{equation*}
\sum_{p \leq X} f(p) = (\tau + o(1))\frac{X}{\log X}.
\end{equation*}
\end{itemize}
Then, we have
\begin{equation*}
\sum_{n \leq X}f(n) = \lbrb{\frac{1}{e^{\gamma \tau}\Gamma(\tau)} + o(1)} \frac{X}{\log X} \prod_{p \leq X} \sum_{v = 0}^{\infty} \frac{f(p^v)}{p^v}.
\end{equation*}
Here $\Gamma$ is the gamma function and $\gamma$ is the Euler--Mascheroni constant.
\end{lemma}

\begin{lemma}[{\cite[Satz 1.2.2]{Wir2}}] \label{Wirsing2}
Let $f$ be a non-negative multiplicative function such that
\begin{itemize}
\item there exists a constant $c$ satisfying $f(p) \leq c$ for all primes $p$,
\item there exists a real number $\tau > 0$ satisfying
\begin{equation*}
\sum_{p \leq X} f(p) = (\tau + o(1))\frac{X}{\log X}.
\end{equation*}
\end{itemize}
Then, for $g$ with $|g(n)| \leq f(n)$, we have
\begin{equation*}
\sum_{n \leq X} g(n) =  \lbrb{\prod_p \lbrb{1+ \frac{g(p)}{p} + \frac{g(p^2)}{p^2} + \cdots } \lbrb{1+ \frac{f(p)}{p} + \frac{f(p^2)}{p^2} + \cdots }^{-1}  + o(1)}  \lbrb{\sum_{n \leq X} f(n)}.
\end{equation*}
\end{lemma}

The following lemma is a variant of \cite[Theorem 1]{Wil}.

\begin{lemma}
  \label{lem:williams}
  Let $\xi$ be an $n$-th root of unity for some positive integer $n$,
  and let $a$ and $m$ be positive integers satisfying
  $(n, m) = 1$ and $(a,m) = 1$. Then,
\begin{equation*}
  \prod_{\substack{ p \leq X \\ p \equiv a \pmod{m} }} \lbrb{1 - \frac{\xi}{p}}
     = c_{a, m, \xi} \zeta_{a, m, \xi, X} \left| \prod_{p \leq X} \lbrb{1 - \frac{\xi}{p} } \right|^{\frac{1}{\phi(m)} } 
     + O \lbrb{\frac{1}{(\log X)^{\frac{1}{\phi(m)}}}  \left| \prod_{p \leq X} \lbrb{1 - \frac{\xi}{p} } \right|^{\frac{1}{\phi(m)} } } 
\end{equation*}
for a non-zero constant $c_{a, m, \xi}$ and a $\phi(m)$-th root of unity $\zeta_{a, m, \xi, X}$.
\end{lemma}

\begin{proof}
Let $\psi$ be a non-principal Dirichlet character of order $m$.
There is a constant $b_{\psi}$ such that
\begin{equation*}
\sum_{p \leq X} \frac{\psi(p)}{p} = b_{\psi} + O_{\psi}\lbrb{\frac{1}{\log X}}.
\end{equation*}
Then,
\begin{align*}
-\sum_{p \leq X} \log \lbrb{1 - \frac{\xi\psi(p)}{p}} 
&= \xi \sum_{p\leq X} \frac{\psi(p)}{p} + \sum_{p }\lbrb{ \sum_{k=2}^\infty \frac{\xi^k\psi(p)^k}{kp^k}} + O_{\psi}\lbrb{\frac{1}{\log X}}
= b'_{\psi, \xi} + O_{\psi}\lbrb{\frac{1}{\log X}},
\end{align*}
for some constant $b'_{\psi, \xi}$.  Therefore, by taking exponential
to the above equation, we see that there is a constant
$c_{\psi, \xi} \neq 0$ such that
\begin{equation} \label{eqn:Euler-xichi}
\prod_{p \leq X} \lbrb{1 - \frac{\xi\psi(p)}{p}}^{-1} = c_{\psi, \xi} + O_{\psi}\lbrb{\frac{1}{\log X}}.
\end{equation}

Let $k_{\psi, \xi}$ be a completely multiplicative function defined by
\begin{equation*}
k_{\psi, \xi}(p) := p\lbrb{1- \lbrb{1 - \frac{\xi \psi(p)}{p}} \lbrb{1 - \frac{\xi}{p}}^{-\psi(p)}},
\end{equation*}
for each prime $p$.
By following \cite{Wil}, we have
\begin{equation*} \label{eqn:boundkxi}
|k_{\psi, \xi}(p)|
\leq \frac{1}{p} + \sum_{n=3}^{\infty} \frac{1}{p^{n-1}}\frac{n-1}{n} \leq \frac{1}{p-1}.
\end{equation*}
Therefore,
\begin{equation*}
\sum_{p} \bigg|\sum_{n=1}^{\infty} \lbrb{\frac{k_{\psi, \xi}(p)}{p^s}}^n\bigg|
\leq \sum_p \sum_{n=1}^{\infty} \lbrb{\frac{1}{(p-1)p^\sigma}}^n = \sum_{p} \frac{1}{p^{\sigma+1} - p^\sigma - 1}
\end{equation*}
is finite for $\sigma  = \mathrm{Re}(s) > 0$. In the same region,
\begin{equation*}
\prod_p \lbrb{1 - \frac{k_{\psi, \xi}(p)}{p^s}}^{-1} = \prod_p \lbrb{1 + \sum_{n=1}^{\infty} \lbrb{\frac{k_{\psi, \xi}(p)}{p^s}}^n}
\end{equation*}
converges absolutely.  Hence,
\begin{equation*}
K_{\xi}(s, \psi) := \sum_{n=1}^{\infty} \frac{k_{\psi, \xi}(n)}{n^s}
\end{equation*}
converges absolutely and has Euler product on $\sigma > 0$.
Therefore, there exists $c_\psi \neq 0$ such that
\begin{equation} \label{eqn:Euler-kchixi}
\prod_{p \leq X} \lbrb{1 - \frac{k_{\psi, \xi}(p)}{p}}^{-1} =  c_\psi
+ O_\psi\lbrb{\frac{1}{\log X}}.
\end{equation}
By the orthogonality of characters,
\begin{equation*}
  \prod_{\substack{ p \leq X \\ p \equiv a \pmod{m} }} \lbrb{1 -
    \frac{\xi}{p}}^{\phi(m)}
  = \prod_{\psi} \lbrb{\prod_{p \leq X} \lbrb{1 -
      \frac{\xi}{p}}^{\psi(p)} }^{\overline{\psi(a)}},
\end{equation*}
where $\phi$ is the usual Euler phi function.
Then by Equations \eqref{eqn:Euler-xichi} and \eqref{eqn:Euler-kchixi}
and the definition of $k_{\psi, \xi}$, we have
\begin{align*}
  \prod_{p \leq X} \lbrb{1 - \frac{\xi}{p} }^{\psi(p)}
  &= \prod_{p \leq X} \lbrb{1 - \frac{\xi \psi(p)}{p} } \prod_{p \leq
    X} \lbrb{1 - \frac{k_{\psi, \xi}(p)}{p}}^{-1} \\
  &= \lbrb{c_{\psi,\xi} + O\lbrb{\frac{1}{\log X}} } \lbrb{c_\psi +
    O\lbrb{\frac{1}{\log X}}}
    = c_{\psi,\xi}c_{\psi} + O\lbrb{\frac{1}{\log X}},
\end{align*}
for a non-principal character $\psi$.
Therefore,
\begin{align*}
  \prod_{\substack{ p \leq X \\ p \equiv a \pmod{m} }} \lbrb{1 -
    \frac{\xi}{p}}^{\phi(m)} &= \prod_{\substack{p \leq X \\ p \nmid m}}\lbrb{1 - \frac{\xi}{p}} 
    \cdot \prod_{\psi \neq \psi_0} \lbrb{ c_{\psi, \xi} c_{\psi} + O\lbrb{\frac{1}{\log X}} }^{\overline{\psi(a)}} \\
    &= c'_{a, m, \xi} \cdot \prod_{p \leq X} \lbrb{1 - \frac{\xi}{p}} + O\lbrb{\frac{1}{\log X} \prod_{p \leq X} \lbrb{1 - \frac{\xi}{p}} } .
\end{align*}
for some non-zero constant $c'_{a, m, \xi}$ depending only on $a, m$ and $\xi$. This proves the lemma by choosing $c_{a, m, \xi} = c_{a, m , \xi}'^{\frac{1}{\phi(m)}}$.
\end{proof}

Similar to the equation (\ref{eqn:Euler-xichi}), one can deduce that for a complex number $\xi$ such that $|\xi| < 2$,
\begin{equation*}
\prod_{p \leq X} \lbrb{1 - \frac{\xi}{p}}^{-1} = c_{0, \xi} (\log X)^{\xi} + O((\log X)^{\xi - 1})
\end{equation*}
for some non-zero complex constant $c_{0, \xi}$. Since $\prod_{p \leq X}(1 - \xi/p)(1 - \xi/p)^{-1} = 1$, for a $\xi$ with $\Re(\xi) > 0$ we have
\begin{equation*}
\prod_{p \leq X} \lbrb{1 - \frac{\xi}{p}} = c_{\xi}(\log X)^{-\xi} +
 O\lbrb{\frac{1}{(\log X)^{1+\xi}}}
\end{equation*}
for a non-zero constant $c_{\xi}$.
For $|\xi|=1$, by a standard argument on the Euler product of multiplicative functions, we have 
\begin{equation*}
\prod_{p \leq X} \lbrb{1 - \frac{\xi^2}{p^2}} = c_{1, \xi} + O\lbrb{\frac{1}{X}}
\end{equation*}
for a non-zero constant $c_{1,\xi}$.
Hence, for $\xi$ also satisfying $\Re(\xi) \geq 0$ we have
\begin{equation*}
\prod_{p \leq X} \lbrb{1 + \frac{\xi}{p}} = c_{-\xi}(\log X)^{\xi}  + O((\log X)^{-1 + \xi}).
\end{equation*}
Together with Lemma \ref{lem:williams}, for all $\xi$ with $|\xi| = 1$, there is a non-zero constant $c_{a, m, \xi}''$
and a $\phi(m)$-th root of unity $\zeta_{a, m, \xi, X}$ satisfying 
\begin{align}
\prod_{\substack{ p \leq X \\ p \equiv a \pmod{m} }} \lbrb{1 - \frac{\xi}{p}}
     &= c_{a, m, \xi} \zeta_{a, m \xi, X} |c_{\xi} (\log X)^{-\xi} + O((\log X)^{-1 - \xi})|^{\frac{1}{\phi(m)}}
     + O\lbrb{\left( \log X \right)^{- \frac{\cos(\xi)+ 1}{\phi(m)} } }
     \nonumber \\ 
     &= c_{a, m, \xi}''\zeta_{a, m, \xi, X} \left( \log X \right)^{- \frac{\cos(\xi)}{\phi(m)} } 
     + O\lbrb{\left( \log X \right)^{- \frac{\cos(\xi)+ 1}{\phi(m)} } }.
      \label{eqn:app williams}
\end{align}
Here we consider $\xi$ as an angle on the complex plane since $|\xi| = 1$.

\begin{proposition}
  \label{arithmeticprogression}
  Let $a,b,m$ and $n$ be integers such that $(a,m) = (b,n)=(m,n)=1$,
  $\Omega(X)$ be the set of square-free positive integers $x$ whose
  size is less than $X$ and each prime divisor of $x$ is
  congruent to $b$ modulo $n$. Then there is a constant $c \neq 0$
  such that
\begin{equation*}
 |\Omega(X)| = c(1 + o(1)) \frac{X}{(\log X)^{1-\frac{1}{\phi(n)} }}.
\end{equation*}
Let $\Omega_{a, m}(X)$ be a set of elements of $\Omega(X)$
which is equal to $a$ modulo $m$.  Then
\begin{equation*}
  |\Omega_{a,m}(X)| = \frac{1}{m}|\Omega(X)| + O\lbrb{ \frac{X}{\log
      X} }.
\end{equation*}
\end{proposition}

\begin{proof}

We recall \cite[Theorem 1]{Wil}, that is
\begin{equation} \label{Wilthm}
 \prod_{\substack{p\leq X \\ p \equiv b }}\lbrb{1- \frac{1}{p}}
 =\lbrb{e^{-\gamma} \frac{n}{\phi(n)} \prod_{\psi \neq \psi_0}
   \lbrb{\frac{K_1(1, \psi)}{L(1, \psi)}}^{\overline{\psi}(b)}
 }^{\frac{1}{\phi(n)}} \frac{1}{(\log X)^{\frac{1}{\phi(n)}}} +
 O\lbrb{\frac{1}{(\log X)^{\frac{1}{\phi(n) -1}}}} ,
\end{equation}
where $\psi_0$ is the principal Dirichlet character of modulus $m$.
Following the proof of \cite[Lemma 11]{BFL}, define a
multiplicative function $f$ by, for each prime $p$,
\begin{equation*}
f(p) = \left\{ \begin{array}{cc}
1 & p \equiv b \pmod{n}, \\
0 & \textrm{otherwise.}
\end{array} \right.
\end{equation*}
and $f(p^v) = 0$ for all $v \geq 2$. Then, we have
$\sum_{k \leq X}f(k) = |\Omega(X)|$.  By the prime number theorem for
arithmetic progressions, this multiplicative function $f$ satisfies
the condition of Lemma \ref{Wirsing} for $c_2 = 0$ and
$\tau = 1/\phi(n)$.  Therefore,
\begin{equation*}
|\Omega(X)| = \lbrb{\frac{1}{e^{\frac{\gamma}{\phi(n)}
    }\Gamma(\phi(n)^{-1})} + o(1)} \frac{X}{\log X} \prod_{\substack{p
    \leq X \\ p \equiv b}} \lbrb{1 + \frac{1}{p}}.
\end{equation*}
Since the product $\prod_{p \le X,\,\, p \equiv b \pmod{n}} (1 -
1/p^2)$ is convergent to a non-zero constant as $X \to \infty$,
Equation \eqref{Wilthm} yields
\begin{equation*}
\prod_{\substack{p \leq X \\ p \equiv b}} \lbrb{1 + \frac{1}{p}} = c_3(1+o(1)) (\log X)^{\frac{1}{\phi(n)}},
\end{equation*}
for some non-zero constant $c_3$.  This gives the first part of
Proposition.

By the orthogonality of characters, we have
\begin{equation*}
|\Omega_{a,m}(X)| = \frac{1}{\phi(m)} \sum_{k \leq X} \sum_{\psi}
\overline{\psi(a)} \psi(k) f(k)
= \frac{1}{\phi(m)} \sum_{\psi} \overline{\psi(a)} \sum_{k \leq X}
\psi(k) f(k),
\end{equation*}
where $\sum_{\psi}$ is taken over Dirichlet characters of modulus $m$.
Therefore, it is enough to show that
\begin{equation*}
\sum_{k \leq X} f(k)\psi(k)  = \begin{cases}
(1+ o(1)) \cdot \frac{\phi(m)}{m} \cdot |\Omega(X)| & \text{ if }\psi
= \psi_0, \\
O\lbrb{ \frac{X}{\log X}  } & \text{ if }\psi \neq \psi_0.
\end{cases}
\end{equation*}
We first consider $(\psi = \psi_0)$-case.  We define
\begin{equation*}
m' = \prod_{ \substack{p \mid m \\ p \equiv b \pmod{n}  }} p.
\end{equation*}
Since the number of elements in $\Omega(X)$ which is divisible by $p$
is exactly $|\Omega(X/p)|$, by the inclusion-exclusion principle, we
have
\begin{equation*}
\sum_{k \leq X} f(k)\psi_0(k)
= \sum_{d \mid m'}\mu(d)|\Omega(\frac{X}{d})| .
\end{equation*}
By the first part of this proposition, we have
\begin{equation*}
\sum_{k \leq X} f(k)\psi_0(k) = (1+o(1))\cdot \lbrb{\sum_{d \mid m'}
  \frac{\mu(d)}{d} }\cdot|\Omega(X)|
= (1 + o(1)) \cdot \frac{\phi(m')}{m'} \cdot |\Omega(X)|.
\end{equation*}

Since $f(k)$ and $f(k)\psi(k)$ satisfy the condition of Lemma
\ref{Wirsing2}, we have
\begin{equation} \label{Elliott}
\sum_{k \leq X} f(k)\psi(k) =
\lbrb{\prod_{p \leq X}\lbrb{\sum_{i=0}^{\infty}\frac{f(p^i)
      \psi(p^i)}{p^i} } \lbrb{\sum_{j=0}^{\infty}\frac{f(p^j)}{p^j}
  }^{-1} +o(1)} \sum_{k \leq X} f(k).
\end{equation}
Now, we will show that
\begin{equation} \label{Dirichletuniform}
\prod_{\substack{ p \leq X \\ p \equiv b \pmod{n}   }}\lbrb{1+
  \frac{\psi(p)}{p} }
\end{equation}
converges to a non-zero complex number as $X$ goes to infinity, for
all non-principal characters $\psi$ of modulus $m$.  
Write
\begin{align*}
\prod_{\substack{p \leq X \\ p \equiv b \pmod{n}}} \lbrb{1 -
  \frac{\psi(p)}{p}}
= \prod_{i=0}^{m-1} \prod_{\substack{p \leq X \\ p \equiv b+in
    \pmod{mn}}} \lbrb{1 - \frac{\psi(b+in)}{p}} .
\end{align*}
Then by (\ref{eqn:app williams}), there are non-zero $c_i$ and $\phi(m)$-th root of unities $\zeta_{i,X}$ such that
\begin{align*}
 \prod_{\substack{p \leq X \\ p \equiv b+in
    \pmod{mn}}} \lbrb{1 - \frac{\psi(b+in)}{p}} &=  c_i \zeta_{i,X} 
    (\log X)^{\frac{-\cos(\arg(\psi(b+in)))}{\phi(m)}} 
    + O \lbrb{ (\log X)^{\frac{-\cos(\arg(\psi(b+in)))-1}{\phi(m)}} 
    }.
\end{align*}
Here we understand $\cos(\arg(0))$ as zero.
Since $\sum_{a = 0}^{m-1} \psi(a) = 0$, we have
\begin{align*}
\prod_{\substack{p \leq X \\ p \equiv b \pmod{n}}} \lbrb{1 -
  \frac{\psi(p)}{p}} &= \left(\prod_{i=0}^{m-1} c_i \zeta_{i, X}\right)
    + O\lbrb{\frac{1}{(\log X)^{\frac{1}{\phi(m)}}}}. 
\end{align*}
Hence it  converges  absolutely, one can show that it also converges to a non-zero complex number.
Since
\begin{equation*}
  \prod_{\substack{p \leq X \\ p \equiv b \pmod{n}}} \left(1 -
    \frac{\psi^2(p)}{p^2}\right)
\end{equation*}
converges to a non-zero constant, the product \eqref{Dirichletuniform}
also converges to a non-zero constant.  Then by (\ref{Elliott}) and
the previous calculations, there is a non-zero constant $c_4$ such
that
\begin{equation*}
\sum_{k \leq X} f(k) \psi(k) = c_4 \frac{1}{(\log
  X)^{\frac{1}{\phi(n)}}} \sum_{k \leq X} f(k),
\end{equation*}
which gives the cases with $(\psi \neq \psi_0)$.
\end{proof}


\begin{proof}[Proof of Theorem \ref{main average}]
  Recall for a real number $X > 0$, $\mathcal{Q}(X)$ is the set of
  elements $x \in O$ with $|x| < X$, represented as the product of
  distinct primary primes of degree two.  Write $d$ as the product of
  primary primes (cf.~Equation \eqref{eq:d-expression}):
  \begin{equation*}
    d = \prod_{i \in I} \pi_i \cdot
    \prod_{j \in J} (-q_j)^{n_j},
  \end{equation*}
  and let
  \begin{equation*}
    d_1 = \prod_{i \in I} \pi_i, \qquad
    d_2 = \prod_{j \in J} (-q_j)^{n_j}, \qquad
    \widehat{d_2} = \prod_{j \in J} (-q_j).
  \end{equation*}
  For each function $f: J \to \{0, 1\}$ we define
  \begin{equation*}
    \cQ_f(X) := \lcrc{Q \in \cQ(X) : \ord_{q_j} Q = f(j)}.
  \end{equation*}
  Then $\mathcal{Q}(X)$ is the disjoint union of $\mathcal{Q}_f(X)$,
  for all $f \in \{0, 1 \}^J$.  For a fixed $f \in \{0, 1\}^J$, define $L = f^{-1}(0)$ and $T = f^{-1}(1)$.  Then we can write
  $d = \prod_{i \in I} \pi_i \cdot \prod_{l \in L} (-q_l)^{n_l} \cdot
  \prod_{t \in T} (-q_t)^{n_t}$.  For $Q \in \mathcal{Q}_f(X)$, we write
  $Q = \prod_{k \in K} (-q_k) \cdot \prod_{t \in T} (-q_t)$.

  By Proposition \ref{prop:determination-of-root-number} and Corollary
  \ref{cor:expanding-d}, we have, for each $i \in I$,
  \begin{multline*}
    w_{\pi_i}(\chi_{dQ}) = \eta \cdot \frac{\pi_i}{|\pi_i|} \cdot
    \overline{\quasym{dQ/\pi_i}{\pi_i}}\cdot
    \overline{\quasym{\overline{\pi_i}^{-1}}{\pi_i}}\cdot
    G(\chi_{dQ}^{\pi_i}) \\
    = \eta \cdot \frac{\pi_i}{|\pi_i|} \cdot
    \overline{\quasym{d/\pi_i}{\pi_i}} \cdot
    \overline{\quasym{\overline{\pi_i}^{-1}}{\pi_i}}\cdot
    G(\chi_{d}^{\pi_i}) \cdot \overline{\quasym{Q}{\pi_i}} =
    w_{\pi_i}(\chi_{d}) \cdot \prod_{k \in K}
    \overline{\quasym{-q_k}{\pi_i}} \cdot \prod_{t \in T}
    \overline{\quasym{-q_t}{\pi_i}}.
  \end{multline*}
  Also for each $k \in K$,
  \begin{equation*}
    w_{q_k}(\chi_{dQ}) = w_{q_k}(\chi_{d_1\cdot(-q_k)}) =
    \overline{\quasym{d_1}{-q_k}} =\prod_{i \in I} \overline{ \quasym{-q_k}{\pi_i} },
  \end{equation*}
  and for $l \in L$,
  \begin{equation*}
    w_{q_l}(\chi_{{d}}) = w_{q_l}(\chi_{{Qd}}).
  \end{equation*}


  Therefore,
  \begin{align}
    \frac{w(\chi_{{dQ}})}{w_2(\chi_{{dQ}})}
    &= -i \prod_{i \in I} w_{\pi_i}(\chi_{dQ}) \prod_{l \in L}
      w_{q_l}(\chi_{dQ})
      \prod_{k \in K} w_{q_k}(\chi_{dQ})
      \prod_{t \in T} w_{q_t}(\chi_{dQ}) \nonumber \\
    &= -i \prod_{i \in I} \lbrb{w_{\pi_i}(\chi_d) \prod_{k \in K}
      \overline{\quasym{-q_k}{\pi_i}} \prod_{t \in T}
      \overline{\quasym{-q_t}{\pi_i}}}
      \lbrb{\prod_{l \in L} w_{q_l}(\chi_d)} \lbrb{ \prod_{t \in T}
      w_{q_t}(\chi_{dQ}) } \prod_{k \in K} \lbrb{\prod_{i \in I}
      \overline{\quasym{-q_k}{\pi_i}}} \nonumber \\
    &= \frac{w(\chi_d)}{w_2(\chi_d)}
      \prod_{i \in I} \lbrb{  \prod_{k \in K}
      \overline{\quasym{-q_k}{\pi_i}}\prod_{t \in T}
      \overline{\quasym{-q_t}{\pi_i}}} \prod_{k \in K}\lbrb{\prod_{i
      \in I} \overline{\quasym{-q_k}{\pi_i}}} \lbrb{\prod_{t \in T}
      \frac{w_{q_t}(\chi_{dQ})}{w_{q_t}(\chi_{d})} } \nonumber \\
    &= \frac{w(\chi_d)}{w_2(\chi_d)}
      \prod_{i \in I}  \prod_{k \in K} {\quadsym{-q_k}{p_i}}
      \prod_{i \in I}\prod_{t \in T} \overline{\quasym{-q_t}{\pi_i}}
      \lbrb{\prod_{t \in T}
      \frac{w_{q_t}(\chi_{dQ})}{w_{q_t}(\chi_{d})}
      }, \label{term:w/w2}
  \end{align}
  where $p_i = \pi_i\overline{\pi_i}$.  Here we use
  $\quasym{-q_{k}}{\pi_{i}}^{2} = \quadsym{-q_{k}}{p_{i}}$.
We note that for all $Q \in \cQ_f(X)$, the terms
\begin{equation*}
\prod_{i \in I}\prod_{t \in T} \overline{\quasym{-q_t}{\pi_i}}
\lbrb{\prod_{t \in T} \frac{w_{q_t}(\chi_{dQ})}{w_{q_t}(\chi_{d})} }
\end{equation*}
in (\ref{term:w/w2}) are equal.

As before, let $\mathcal{Q}_{f, a, m}(X)$ be the set of elements in
$\mathcal{Q}_f(X)$ that is equivalent to $a$ modulo $m$, and $P :=
\prod_{i \in I} p_i$.  Then,
\begin{align*}
  \sum_{Q \in \mathcal{Q}_f(X)} \frac{w(\chi_{dQ})}{w_2(\chi_{dQ})}
  &= \sum_{Q \in \mathcal{Q}_f(X)} \frac{w(\chi_d)}{w_2(\chi_d)}
    \prod_{i \in I}  \prod_{k \in K} {\quadsym{-q_k}{p_i}} \prod_{i
    \in I}\prod_{t \in T} \overline{\quasym{-q_t}{\pi_i}}
    \lbrb{\prod_{t \in T} \frac{w_{q_t}(\chi_{dQ})}{w_{q_t}(\chi_{d})}
    }  \\
  &= \frac{w(\chi_{d})}{w_2(\chi_{d})}
    \prod_{i \in I}\prod_{t \in T} \overline{\quasym{-q_t}{\pi_i}}
    \lbrb{\prod_{t \in T} \frac{w_{q_t}(\chi_{dQ})}{w_{q_t}(\chi_{d})}
    }  \lbrb{\sum_{Q \in \mathcal{Q}_f(X)} \prod_{i \in I} \prod_{k \in
    K} \quadsym{-q_k}{p_i}} \\
  &= \frac{w(\chi_{d})}{w_2(\chi_{d})}
    \prod_{i \in I}\prod_{t \in T} \overline{\quasym{-q_t}{\pi_i}}
    \lbrb{\prod_{t \in T} \frac{w_{q_t}(\chi_{dQ})}{w_{q_t}(\chi_{d})}
    } \lbrb{\prod_{i \in I} \prod_{t \in T}
    \overline{\quasym{-q_t}{\pi_i}^{2}}}
    \sum_{Q \in \mathcal{Q}_f(X)} \quadsym{Q}{P} \\
  &= \frac{w(\chi_{d})}{w_2(\chi_{d})}
    \prod_{i \in I}\prod_{t \in T} \overline{\quasym{-q_t}{\pi_i}^3}
    \lbrb{\prod_{t \in T} \frac{w_{q_t}(\chi_{dQ})}{w_{q_t}(\chi_{d})}
    } \lbrb{\sum_{a,+} \sum_{Q \in \mathcal{Q}_{f, a, P}(X)}1
    - \sum_{a,-} \sum_{Q \in \mathcal{Q}_{f, a, P}(X)}1},
\end{align*}
where the sum $(a, +)$ is taken over quadratic residues $a$ modulo $P$
and $(a, -)$ is taken over non-residues.  Since
$n \in \mathcal{Q}_{f, a, P}(X)$ if and only if $n \in \cQ(X)$,
$n \equiv a \pmod{P}$ and $n \not\equiv 0 \pmod{q_j}$ for $q_j$
satisfying $f(j) = 0$, we get
\begin{multline*}
  \sum_{Q \in \mathcal{Q}_f(X)} \frac{w(\chi_{dQ})}{w_2(\chi_{dQ})} \\
  = \frac{w(\chi_{d})}{w_2(\chi_{d})}
\prod_{i \in I}\prod_{t \in T} \overline{\quasym{-q_t}{\pi_i}}
\lbrb{\prod_{t \in T} \frac{w_{q_t}(\chi_d)}{w_{q_t}(\chi_{dQ})} }
\lbrb{\sum_{a,+} \sum_{\substack{b \equiv a (P) \\ b \not\equiv 0
      (q_j) \\ \textrm{ for }f(j) = 0}}|\mathcal{Q}_{b,
    P\widehat{d_2}}(X)|
- \sum_{a,-} \sum_{\substack{b \equiv a (P) \\ b \not\equiv 0 (q_j) \\
    \textrm{ for } f(j) = 0}}|\mathcal{Q}_{b, P\widehat{d_2}}(X)|},
\end{multline*}
where $\mathcal{Q}_{b, P\widehat{d_2}}(X)$ is the subset of
$\mathcal{Q}(X)$ consisting of integers $x$ satisfying
$x \equiv b \pmod{P \widehat{d_2}}$.  Similarly consider
$\Omega_{b, P\widehat{d_2}}(X)$ is the set of positive integers $x$
such that $x \le X$, that each prime divisor of $x$ is congruent to 3
modulo 4, and that $x \equiv b \pmod{P \widehat{d_2}}$.  Then one can
easily see that the map
\begin{equation*}
x \mapsto |x|, \qquad \mathcal{Q}_{b,P\widehat{d_2}}(X) \cup
\mathcal{Q}_{-b,P\widehat{d_2}}(X) \to \Omega_{b,P\widehat{d_2}}(X)
\cup \Omega_{-b,P\widehat{d_2}}(X)
\end{equation*}
is bijective. Since $P\widehat{d_2} \equiv 1\pmod{4}$, $b$ and $-b$
are both quadratic residues or both quadratic non-residues modulo
$P\widehat{d_2}$. Therefore
\begin{equation*}
\sum_{a,+} \sum_{\substack{b \equiv a (P) \\ b \not\equiv 0 (q_j) \\
    \textrm{ for } f(j) = 0}}|\mathcal{Q}_{b, P\widehat{d_2}}(X)| =
\sum_{a,+} \sum_{\substack{b \equiv a (P) \\ b \not\equiv 0 (q_j) \\
    \textrm{ for } f(j) = 0}} |\Omega_{a,P\widehat{d_2}}(X)|.
\end{equation*}
and the same holds for $\sum_{a, -}$.
By Proposition \ref{arithmeticprogression},
\begin{equation*}
\sum_{Q \in \mathcal{Q}_f(X)} \frac{w(\chi_{dQ})}{w_2(\chi_{dQ})}=
\sum_{a,+} \sum_{\substack{b \equiv a (P) \\ b \not\equiv 0 (q_j) \\
    \textrm{ for } f(j) = 0}}|\mathcal{Q}_{b, P\widehat{d_2}}(X)|
- \sum_{a,-} \sum_{\substack{b \equiv a (P) \\ b \not\equiv 0 (q_j) \\
    \textrm{ for } f(j) = 0}}|\mathcal{Q}_{b, P\widehat{d_2}}(X)|  \ll
o(1)\frac{X}{(\log X)^{\frac{1}{2} }}.
\end{equation*}
Therefore,
\begin{align*}
  \sum_{Q \in \cQ(X)} \frac{w(\chi_{dQ})}{w_2(\chi_{dQ})}
  &= \sum_{f \in
    \{0, 1\}^J} \sum_{Q \in \mathcal{Q}_f(X)}
    \frac{w(\chi_{dQ})}{w_2(\chi_{dQ})}
    \ll o(1) \frac{X}{(\log X)^{\frac{1}{2}}},
\end{align*}
so for any $\epsilon > 0$, there is a $X_0$ such that
\begin{equation*}
  \sum_{Q \in \cQ(X)} \frac{w(\chi_{dQ})}{w_2(\chi_{dQ})}  \leq c \cdot
  \epsilon \cdot \frac{X}{ (\log X)^{\frac{1}{2}}},
\end{equation*}
for all $X \geq X_0$, where $c$ is the constant in the statement of
Lemma \ref{arithmeticprogression}.  Hence
\begin{equation*}
  \lim_{X \to \infty} \frac{1}{|\cQ(X)|} \sum_{Q \in \cQ(X)}
  \frac{w(\chi_{{dQ}})}{w_2(\chi_{{dQ}})} \leq
  \lim_{X \to \infty} \frac{1}{|\cQ(X)|} \cdot c \cdot\epsilon\cdot
  \frac{X}{(\log X)^{\frac{1}{2} }} \leq \epsilon,
\end{equation*}
again by Lemma \ref{arithmeticprogression}.
\end{proof}


\section*{Acknowledgement}
\label{sec:acknowledgement}

The first author 
was supported by
Basic Science Research Program
through the National Research Foundation of Korea (NRF) funded
by the Ministry of Education (Grant No. 2018R1C1C1004264). He would like to thank
Professor Peter J.~Cho and Professor Yoonbok Lee for useful
discussion.  He also thanks IBS-CGP for their hospitality and
financial support.  
The second author was supported by
Basic Science Research Program
through the National Research Foundation of Korea (NRF) funded
by the Ministry of Education 
(Grant Nos. 2019R1A6A1A11051177 
and 2020R1I1A1A01074746). 
The third author was supported by IBS-R003-D1.
The second and the third authors also thank UNIST for their hospitality.
Authors thank the referee for the valuable suggestions.

\bibliographystyle{alpha}
\bibliography{root-numbers}

\end{document}